\documentclass[12pt, a4]{amsart}
\usepackage{amscd, amsmath, amssymb}
\usepackage{fullpage}

\theoremstyle{plain}
\newtheorem{thm}{Theorem}[section]

\newtheorem{pro}[thm]{Proposition}
\theoremstyle{definition}

\newtheorem{ex}[thm]{Example}

\numberwithin{equation}{section}

\begin{document}

\title[Perturbed Hammerstein integral equations]{Positive and increasing solutions of perturbed Hammerstein integral equations with derivative dependence}  

\date{}

\author[G. Infante]{Gennaro Infante}
\address{Gennaro Infante, Dipartimento di Matematica e Informatica, Universit\`{a} della
Calabria, 87036 Arcavacata di Rende, Cosenza, Italy}%
\email{gennaro.infante@unical.it}%

\begin{abstract} 
We discuss the existence and non-existence of non-negative, non-decreasing solutions of certain perturbed Hammerstein integral equations with derivative dependence. We present some applications to nonlinear, second order boundary value problems subject to fairly general functional boundary conditions. The approach relies on classical fixed point index theory. 
\end{abstract}



\maketitle

\centerline{\it
Dedicated to Professor Juan J. Nieto on the occasion of his sixtieth birthday.
}

\section{Introduction}
The study of perturbed Hammerstein integral equations often arises in the study of real world phenomena. For example the equation
\begin{equation*}
u(t)= t h(u(1)) + \int_0^1 k(t,s)f(s,u(s))\,ds,
\end{equation*}
occurs when dealing with the solvability of the boundary value problem~(BVP)
\begin{equation}\label{bvp-intro}
u''(t)+f(t,u(t))=0,\ u(0)=0,\ u'(1)=h(u(\beta)).
\end{equation}
The BVP~\eqref{bvp-intro} can be used as a model for the steady-states of heated bar of length~$1$, where the left end is kept at ambient temperature and a controller in the right end adds or removes heat according to the temperature  registered by a sensor placed in a point $\beta$ of the bar. The controller placed in the right end may act in a linear or in a nonlinear manner, depending on the nature of the function $h$. There exists now a (relatively) wide literature on heat-flow problems of this kind, we refer the reader to the papers~\cite{fama, df-gi-jp-prse, hGlkDaC, gijwnodea, gijwems,
 nipime, jwpomona, jwwcna04, webb-therm} for the cases of linear response and to~\cite{gi-caa, gi-tmna, kamkee, kamkee2, kapala, palamides} for the nonlinear cases.

Note that the idea of using perturbed Hammerstein integral equations in order to deal with the existence of solutions of BVPs with nonlinear 
BCs has been used with success in a number of papers, see the manuscripts~\cite{amp, Cabada1, ac-gi-at-tmna, genupa, dfdorjp, Goodrich3, Goodrich4, Goodrich5, Goodrich6, Goodrich7, Goodrich8, Goodrich9, gi-caa, kejde, paola, ya1,ya2} and references therein. In particular, in the recent paper~\cite{Goodrich9}, by means of the classical Krasnosel'ski\u\i{}-Guo fixed point theorem of cone compression/expansion, Goodrich studied the existence of positive solutions of the equation
\begin{equation}\label{Chris}
u(t)={\gamma_1}(t)h_{1}(\alpha_{1}[u])+ h_{2}(\alpha_{2}[u]) +\lambda \int_0^1 k(t,s)f(s,u(s))\,ds,
\end{equation}
where $\lambda$ is parameter and $\phi_{1}, \phi_{2}$ are linear functionals on the space $C[0,1]$ realized as Stieltjes integrals with signed measures, namely
\begin{equation}\label{signed}
\alpha_{i}[u]:=\int_{0}^{1}u(s)\,dA_i(s),
\end{equation}
with $A_i$ a function of bounded variation. The results of~\cite{Goodrich9} complement the earlier ones by the author~\cite{gi-caa}, where only positive measures were employed.

The functional formulation~\eqref{signed} has proven to be particularly useful in order to handle multi-point and integral BCs. For an introduction to nonlocal BCs, we refer the reader to the reviews~\cite{Cabada1, Conti, rma, sotiris, Stik, Whyburn} and the papers~\cite{kttmna, ktejde, Nieto, Picone, jw-gi-jlms}.

On the other hand, in a recent paper~\cite{webb-ejqtde}, Webb gave, using fixed point index theory, a general set-up for  the existence of positive solutions of second order BVPs where linear BCs of the type $\alpha[u']=\int_{0}^{1}u'(s)\,dA(s)$ occur, a particular example being the BVP
\begin{equation*}
u''(t)+f(t,u(t))=0,\ u(0)=0,\ u'(1)=\alpha[u'].
\end{equation*}

Also by means fixed point index theory, Zang and co-authors~\cite{zang} discussed the existence of positive, increasing solutions of the BVP
\begin{equation*}
u''(t)+f(t,u(t),u'(t))=0,\ u(0)=\alpha[u],\ u'(1)=0,
\end{equation*}
where $\alpha[u]$ is a linear, bounded functional on the space $C[0,1]$. 

Nonlinear functional BCs were investigated by Mawhin et al. in~\cite{Mawhin}, where the authors prove, by means of degree theory, the existence of a solution of a system of BVPs which, in the scalar case, reduces to 
\begin{equation*}
u''(t)+f(t,u(t),u'(t))=0,\ u(0)=a,\ u'(1)=N[u'],
\end{equation*}
here $a$ is a fixed number and $N$ is a compact functional defined on the space $C[0,1]$.

Here we study an integral equation related to~\eqref{Chris}, where we allow a dependence in the derivative of the nonlinearity $f$ and we allow the (not necessarily linear) functionals to act on the space $C^1[0,1]$, namely
\begin{equation}\label{perhamm-intro}
u(t)=\eta_1{\gamma_1}(t)h_1[u]+ \eta_2{\gamma_2}(t)h_2[u]  +\lambda \int_0^1 k(t,s)f(s,u(s), u'(s))\,ds,
\end{equation}
where $h_1, h_2$ are suitable compact functionals on the space $C^1[0,1]$ and $\eta_1, \eta_2, \lambda$ are non-negative parameters.  Multi-parameter problems of this kind have been studied recently by the author~\cite{gi-tmna} in the context of systems of elliptic equations (without gradient dependence) subject to functional BCs. Here, in the spirit of the paper~\cite{gi-tmna}, we provide existence and non-existence results for the equation~\eqref{perhamm-intro} that take into account the parameters $\eta_1, \eta_2, \lambda$. One advantage of considering the functionals in the space $C^1[0,1]$ is that it allows us to consider an interplay between function and derivative dependence in the BCs, this is illustrated in the examples of Section 3. Our methodology involves the classical fixed point index for the existence result and an elementary argument for the non-existence
result.

As an application we discuss the solvability of the BVP
\begin{equation*}
u''(t)+\lambda f(t,u(t),u'(t))=0,\ u(0)=\eta_1h_1[u],\ u'(1)=\eta_2h_2[u],
\end{equation*}
and illustrate, in two examples, 
how our methodology can be used in presence of nonlinear functionals that involve also nonlocal conditions. 

\section{Main results}
In this Section we  study the existence and non-existence of solutions of the perturbed Hammerstein equation of the type
\begin{equation}\label{perhamm}
u(t)=\eta_1{\gamma_1}(t)h_1[u]+ \eta_2{\gamma_2}(t)h_2[u]  +\lambda \int_0^1 k(t,s)f(s,u(s), u'(s))\,ds:=Tu(t).
\end{equation}
Throughout the paper we make the following assumptions on the terms that occur in~\eqref{perhamm}. 
\begin{itemize}
\item[$(C_1)$] $k:[0,1] \times[0,1]\rightarrow [0,+\infty)$ is measurable and continuous in $t$ for almost every  (a.e.) ~$s$,
that is, for every $\tau\in [0,1] $ we have
\begin{equation*}
\lim_{t \to \tau} |k(t,s)-k(\tau,s)|=0 \ \text{for a.e.}\ s \in [0,1] ,
\end{equation*}{}
furthermore there exist a function $\Phi \in L^{1}(0,1)$ such that
$0\leq  k(t,s)\leq \Phi(s)$ for $t \in [0,1]$ and a.e.~$s\in [0,1]$.

\item[$(C_2)$] The partial derivative $\partial_t k(t,s)$ is non-negative and continuous in $t$ for a.e.~$s$ and there exists $\Psi \in L^{1}(0,1)$ such that
$0\leq  \partial_t k(t,s) \leq \Psi(s)$ for $t \in [0,1]$ and a.e.~$s\in [0,1]$.

\item[$(C_3)$]  $f:[0,1]\times  [0,+\infty) \times [0,+\infty) \to [0,+\infty)$ is continuous.

\item[$(C_4)$]  $\gamma_1, \gamma_2 \in C^1 [0,1] $ and $\gamma_1 (t), \gamma_2 (t), \gamma'_1 (t), \gamma'_2 (t)\geq 0\ge 0\ \text{for every}\ t\in [0,1]$.

\item[$(C_5)$] $\eta_1, \eta_2, \lambda \in [0,+\infty)$.
\end{itemize}
Due to the hypotheses above, we use the space $C^1[0,1]$ endowed with the norm $$\|u\|:=\max\{\|u\|_\infty, \|u'\|_\infty\},$$ where $\|u\|_\infty:=\max_{t\in[0,1]}|u(t)|$. 

We recall that a cone $K$ in a real Banach space $X$ is a closed convex set such that $\lambda x\in K$ for every $x \in K$ and for all $\lambda\geq 0$ and satisfying $K\cap (-K)=\{0\}$. Here, in order to discuss the solvability of~\eqref{perhamm}, we work in the cone of non-negative, non-decreasing functions
$$
P:=\{u\in C^1[0,1]:\ u(t),u'(t)\ge 0\ \text{for every}\ t\in [0,1] \},
$$
and we require the nonlinear functionals $h_1,h_2$ to act positively on the cone $P$ and to be compact, that is:
\begin{itemize}
\item[$(C_6)$] $h_1,h_2: P \to [0,+\infty)$ are continuous and map bounded sets into bounded sets.
\end{itemize}

We make use of the following basic properties of the fixed point index, we refer the reader to~\cite{amann, guolak} for more details.

\begin{pro}\cite{amann, guolak} Let $K$ be a cone in a real Banach space $X$ and let
$D$ be an open bounded set of $X$ with $0 \in D_{K}$ and
$\overline{D}_{K}\ne K$, where $D_{K}=D\cap K$.
Assume that $\tilde{T}:\overline{D}_{K}\to K$ is a compact map such that
$x\neq \tilde{T}x$ for $x\in \partial D_{K}$. Then the fixed point index
 $i_{K}(\tilde{T}, D_{K})$ has the following properties:
\begin{itemize}
\item[$(1)$] If there exists $e\in K\setminus \{0\}$
such that $x\neq \tilde{T}x+\lambda e$ for all $x\in \partial D_K$ and all
$\lambda>0$, then $i_{K}(\tilde{T}, D_{K})=0$.

\item[$(2)$] If $\tilde{T}x \neq \lambda x$ for all $x\in
\partial D_K$ and all $\lambda > 1$, then $i_{K}(\tilde{T}, D_{K})=1$.

\item[(3)] Let $D^{1}$ be open in $X$ such that
$\overline{D^{1}}_{K}\subset D_K$. If $i_{K}(\tilde{T}, D_{K})=1$ and $i_{K}(\tilde{T},
D_{K}^{1})=0$, then $\tilde{T}$ has a fixed point in $D_{K}\setminus
\overline{D_{K}^{1}}$. The same holds if 
$i_{K}(\tilde{T}, D_{K})=0$ and $i_{K}(\tilde{T}, D_{K}^{1})=1$.
\end{itemize}
\end{pro}

We define the set 
$$
P_{\rho}:=\{u\in P: \|u\|<\rho\}
$$
and the quantities
$$
\overline{f}_{\rho}:=\max_{[0,1]\times [0,\rho]^2}f(t,u,v),\quad \underline{f}_{\rho}:=\min_{[0,1]\times [0,\rho]^2}f(t,u,v),\quad H_{i,\rho}:=\sup_{u\in \partial P_{\rho}} h_{i}[u],
$$
$$
K:=\int_0^1 k(1,s)\,ds,\quad K^*:=\sup_{t\in[0,1]}\int_0^1 \partial_t k(t,s)\,ds.
$$

With these ingredients we can state the following existence and localization result.
\begin{thm}\label{thmsol}
Assume there exist $r,R\in (0,+\infty)$, with $r<R$ such that the following two inequalities are satisfied:
\begin{equation}\label{idx1}
\max\Bigl\{\lambda \overline{f}_R K+\sum_{i=1}^{2}\eta_i{\gamma_i}(1)H_{i,R},\  \lambda \overline{f}_R  K^{*}+\sum_{i=1}^{2} \eta_i\|{\gamma'_i}\|_{\infty} H_{i,R} \Bigr \}\leq R,
\end{equation}
\begin{equation}\label{idx0}
\lambda\underline{f}_r \min\{  K,    K^{*} \}\geq r.
\end{equation}

Then the equation~\eqref{perhamm} has a solution $u\in P$ such that $$r\leq \|u\| \leq R.$$
\end{thm}

\begin{proof}
It is routine to prove that, under the assumptions $(C_{1})-(C_{6})$, the operator $T$ maps $P$ into $P$ and is compact.

If $T$ has a fixed point either on $\partial {P_r}$ or $\partial {P_R}$ we are done. 

Assume now that $T$ is fixed point free on $\partial {P_r}\cup\partial {P_R}$, we are going to prove that $T$ has a fixed point in 
$ P_R\setminus \overline {P_r}$.

We firstly prove that 
 $
\sigma  u\neq Tu\ \text{for every}\ u\in \partial P_{R}\
\text{and every}\  \sigma >1.
$
If this does not hold, then there exist $u\in \partial P_{R}$ and $\sigma >1$ such that $\sigma  u= Tu$. 
Note that if $\| u\| = R$ either $\| u\|_{\infty} = R$ or $\| u'\|_{\infty} = R$. 

Assume that $\| u\|_{\infty} = R$. In this case we obtain, for $t\in [0,1]$,
\begin{multline}\label{ineq1}
\sigma u(t)=\eta_1{\gamma_1}(t)h_1[u]+ \eta_2{\gamma_2}(t)h_2[u]  +\lambda \int_0^1 k(t,s)f(s,u(s), u'(s))\,ds \\ 
\leq \eta_1{\gamma_1}(1)H_{1,R}+ \eta_2{\gamma_2}(1)H_{2,R}+\lambda \overline{f}_{R} \int_0^1 k(1,s)\,ds\leq R.
\end{multline}
Taking the supremum for $t\in [0,1]$ in~\eqref{ineq1} gives $\sigma \leq 1$, a contradiction.

Assume that $\| u'\|_{\infty} = R$. In this case we obtain, for $t\in [0,1]$,
\begin{multline}\label{ineq2}
\sigma u'(t)=\eta_1{\gamma'_1}(t)h_1[u]+ \eta_2{\gamma'_2}(t)h_2[u]  +\lambda \int_0^1 \partial_t k(t,s)f(s,u(s), u'(s))\,ds \\ 
\leq \eta_1\|{\gamma'_1}\|_{\infty} H_{1,R}+ \eta_2\|{\gamma'_2}\|_{\infty} H_{2,R}+\lambda \overline{f}_{R}  \int_0^1 \partial_t k(t,s)\,ds\leq R.
\end{multline}
Taking the supremum for $t\in [0,1]$ in~\eqref{ineq2} yields $\sigma \leq 1$, a contradiction.

Therefore we have $i_{P}(T, P_R)=1.$

We now consider the function $g(t):= t$ in $[0,1]$, note that $g\in P$. We show that 
\begin{equation*} 
u\neq Tu+\sigma g\ \text{for every}\ u\in \partial P_{r}\ \text{and every}\ \sigma  >0.
\end{equation*}
If not, there exists $u\in \partial P_{r}$ and $\sigma  >0$ such that
$
u= Tu+\sigma  g . 
$

Assume that $\| u\|_{\infty} = r$. In this case we obtain, for $t\in [0,1]$,
\begin{multline}\label{ineq3}
 u(t)= \eta_1{\gamma_1}(t)h_1[u]+ \eta_2{\gamma_2}(t)h_2[u]   \\ 
+ \lambda \int_0^1 k(t,s)f(s,u(s), u'(s))\,ds +\sigma t
\geq \lambda \underline{f}_{r} \int_0^1  k(t,s) \,ds+\sigma t.
\end{multline}
Taking the supremum for $t\in [0,1]$ in~\eqref{ineq3} gives $r\geq r+ \sigma$, a contradiction.

Assume that $\| u'\|_{\infty} = r$. In this case we obtain, for $t\in [0,1]$,
\begin{multline}\label{ineq4}
 u'(t)=\eta_1{\gamma'_1}(t)h_1[u]+ \eta_2{\gamma'_2}(t)h_2[u]  \\ 
+\lambda \int_0^1 \partial_t k(t,s)f(s,u(s), u'(s))\,ds + \sigma \geq  \lambda \underline{f}_{r}  \int_0^1 \partial_t k(t,s)\,ds +\sigma.
\end{multline}
Taking the supremum for $t\in [0,1]$ in~\eqref{ineq4} yields $r\geq r+ \sigma$, a contradiction.

Thus we obtain $i_{P}(T, P_{r})=0.$

Therefore we have
$$i_{P}(T, P_R \setminus \overline{P_r})=i_{P}(T, P_R )-i_{P}(T, P_r )=1,$$
which proves the result.
\end{proof}
We now prove, by an elementary argument, a non-existence result.
\begin{thm}\label{nonexthm}
Assume that there exist $\tau, \xi_1,  \xi_2\in (0,+\infty)$ such that
\begin{equation}
0\leq f (t,u,v)\leq \tau u,\ \text{for every}\ (t,u,v)\in [0,1]\times[0,\infty)^2,
\end{equation}
\begin{equation}
h_i[u]\leq \xi_i \|u\|_{\infty}, \ \text{for every}\ u \in P\ \text{and}\ i=1,2,
\end{equation}
\begin{equation}\label{nonexineq}
\lambda \tau K+\sum_{i=1}^{2}\eta_i \xi_i{\gamma_i}(1)<1.
 \end{equation}
Then the equation~\eqref{perhamm} has at most the zero solution in $P$. 
\end{thm}
\begin{proof}
Assume that there exist $u\in P\setminus {0}$ such that $u$ is a fixed point for $T$. Then $\|u\|_{\infty}=\rho$, for some $\rho>0$.
Then we have
\begin{multline}\label{ineq5}
 u(t)=\eta_1{\gamma_1}(t)h_1[u]+ \eta_2{\gamma_2}(t)h_2[u]  +\lambda  \int_0^1 k(t,s)f(s,u(s), u'(s))\,ds \\ 
\leq \eta_1{\gamma_1}(1)h_1[u]+ \eta_2{\gamma_2}(1)h_2[u]+\lambda \tau \int_0^1 k(1,s) u(s)\,ds\\
\leq \eta_1{\gamma_1}(1)\xi_1\|u\|_{\infty}+ \eta_2{\gamma_2}(1)\xi_2\|u\|_{\infty}+\lambda \tau  \|u\|_{\infty}\int_0^1 k(1,s) \,ds.
\end{multline}
Taking the supremum for $t\in [0,1]$ in~\eqref{ineq5} gives $\rho < \rho$, a contradiction.
\end{proof}
\section{Two examples}
We now illustrate the applicability of the results of Section 2. In particular we focus on the BVP 
\begin{equation}\label{bvpex}
u''(t)+\lambda f(t,u(t),u'(t))=0,\ u(0)=\eta_1h_1[u],\ u'(1)=\eta_2h_2[u]. 
\end{equation}
It is routine to show (for some details, see for example~\cite{gi-mc})
 that the solutions of~\eqref{bvpex} can be written in the form 
\begin{equation*}
u(t)=\eta_1{\gamma_1}(t)h_1[u]+ \eta_2{\gamma_2}(t)h_2[u] +\lambda \int_0^1 k(t,s)f(s,u(s), u'(s))\,ds,
\end{equation*}
where the kernel $k$ is the Green's function associated to the right focal BCs
$$u(0)=u'(1)=0,$$ namely
\begin{equation*}
k(t,s)=\begin{cases}
  s & \text{if}\;s \leq t,\\
  t &\text{if}\; s>t,
\end{cases}
\end{equation*}
and ${\gamma_1}(t)=1$ and ${\gamma_2}(t)=t$ are solutions of the BVPs
$$
{\gamma_1}''(t)=0,\ {\gamma_1}(0)=1,\ {\gamma_1}'(1)=0. 
$$
$$
{\gamma_2}''(t)=0,\ {\gamma_2}(0)=0,\ {\gamma_2}'(1)=1. 
$$
In this case we have 
\begin{equation*}
{\gamma_1}'(t)=0,\quad {\gamma_2}'(t)=1\quad \text{and}\quad
\partial_t k(t,s)=\begin{cases}
  0 & \text{if}\;s \leq t,\\
  1 &\text{if}\; s>t.
\end{cases}
\end{equation*}
Therefore the assumptions $(C_1), (C_2)$ and $(C_4)$ are satisfied with $\Phi(s)=s$ and $\Psi(s)=1$.
By direct calculation we have $K=\frac{1}{2}$ and $K^*=1$.
\begin{ex}
Let us consider the BVP
\begin{equation}\label{bvpex1}
u''(t)+{\lambda e^{t(u(t)+u'(t))}}=0,\ u(0)=\eta_1h_1[u],\ u'(1)=\eta_2h_2[u], 
\end{equation}
where
$$
h_{1}[u]=u(1/4)+(u'(3/4))^2,\ h_{2}[u]=\int_{0}^1u^3(s)+u'(s)\,ds.
$$
Let us fix $r=1/20$ and $R=1$, then we have 
$$\overline{f}_{1}=e^2, \underline{f}_{\frac{1}{20}}=1, H_{1,1},H_{2,1}\leq 2.$$
Therefore the condition~\eqref{idx1} is satisfied if
\begin{equation}\label{idx1ex}
\max\Bigl\{\lambda \frac{e^2}{2}+2\eta_1+2\eta_2,\  \lambda e^2 +2\eta_2 \Bigr \}\leq 1,
\end{equation}
and the condition~\eqref{idx0} reads 
\begin{equation}\label{idx0ex}
\lambda\geq \frac{1}{10}.
\end{equation}

For the range of parameters that satisfy the inequalities~\eqref{idx1ex}-\eqref{idx0ex}, Theorem~\ref{thmsol} provides the existence of at least a nondecreasing, nonnegative solution $u$ of the BVP~\eqref{bvpex1} with $1/20\leq \|u\| \leq 1$; this occurs, for example, for $\lambda=1/10, \eta_1=1/11, \eta_2=1/12$.
\end{ex}
\begin{ex}
Let us now consider the BVP
\begin{equation}\label{bvpex2}
u''(t)+{\lambda u(t)(2-t\sin (u(t)u'(t)))}=0,\ u(0)=\eta_1h_1[u],\ u'(1)=\eta_2h_2[u], 
\end{equation}
where
$$
h_{1}[u]=u(1/4)\cos^2 (u'(3/4)),\ h_{2}[u]=u(3/4)\sin^2 (u'(1/4)).
$$
In this case we may take $\tau=3, \xi_1=\xi_2=1$. Then the condition~\eqref{nonexineq} required by Theorem~\ref{nonexthm} reads
\begin{equation}\label{bvpex2rg}
\frac{3}{2}\lambda+\eta_1+\eta_2<1.
\end{equation}

For the range of parameters that satisfy the inequality~\eqref{bvpex2rg}, Theorem~\ref{nonexineq} guarantees that the only possible solution in $P$ of the BVP~\eqref{bvpex2} is the trivial one; this occurs, for example, for $\lambda=1/3, \eta_1=1/4, \eta_2=1/5$.
\end{ex}

\section*{Acknowledgement}
G. Infante was partially supported by G.N.A.M.P.A. - INdAM (Italy).


\begin{thebibliography}{xxx}

\bibitem{amp} E. Alves, T. F. Ma and M. L. Pelicer,
 Monotone positive solutions for a fourth order equation with nonlinear boundary conditions,
\textit{Nonlinear Anal.}, \textbf{71} (2009), 3834--3841.

\bibitem{amann} H. Amann, Fixed point equations and nonlinear
eigenvalue problems in ordered Banach spaces, \textit{SIAM. Rev.},
\textbf{18} (1976), 620--709.

\bibitem{Cabada1} A. Cabada, An overview of the lower and upper solutions method with nonlinear boundary value conditions, \textit{Bound. Value Probl.} (2011), Art. ID 893753, 18 pp.

\bibitem{ac-gi-at-tmna} A. Cabada, G. Infante and F. A. F. Tojo, Nonzero solutions of perturbed Hammerstein integral equations with deviated arguments and applications, \textit{Topol. Methods Nonlinear Anal.}, \textbf{47} (2016), 265--287. 

\bibitem{genupa}
F. Cianciaruso, G. Infante and P. Pietramala, Solutions of perturbed Hammerstein integral equations with applications, \textit{Nonlinear Anal. Real World Appl.}, \textbf{33} (2017), 317--347.

\bibitem{Conti} R. Conti,
Recent trends in the theory of boundary value problems for ordinary differential equations,
\textit{Boll. Un. Mat. Ital.}, \textbf{22} (1967), 135--178.

\bibitem{fama}
 H. Fan and R. Ma, Loss of positivity in a nonlinear second order ordinary differential equations, \textit{Nonlinear Anal.}, \textbf{71} (2009), 437--444.

\bibitem{df-gi-jp-prse}
D. Franco, G. Infante and J. Per\'an, A new criterion for the existence of multiple solutions in cones, \textit{Proc. Roy. Soc. Edinburgh Sect. A}, {\bf 142} (2012), 1043--1050.

\bibitem{dfdorjp} D. Franco, D. O'Regan and J. Per\'an,
Fourth-order problems with nonlinear boundary conditions,
\textit{J. Comput. Appl. Math.}, \textbf{174} (2005), 315--327.

\bibitem{hGlkDaC}
H. Garai, L. K Dey and A. Chanda, 
Positive solutions to a fractional thermostat model in Banach spaces via fixed point results,
\textit{J. Fixed Point Theory Appl.}, {\bf  20} (2018),  Art. 106, 24 pp. 


\bibitem{Goodrich3}
C. S. Goodrich, On nonlocal BVPs with nonlinear boundary conditions with asymptotically sublinear or superlinear growth, \textit{Math. Nachr.}, \textbf{285} (2012), 1404--1421.

\bibitem{Goodrich4}
C. S. Goodrich, Positive solutions to boundary value problems with nonlinear boundary conditions, \textit{Nonlinear Anal.}, \textbf{75} (2012), 417--432. 

\bibitem{Goodrich5}
C. S. Goodrich, On nonlinear boundary conditions satisfying certain asymptotic behavior, \textit{Nonlinear Anal.}, \textbf{76} (2013), 58--67. 

\bibitem{Goodrich6}
C. S. Goodrich,
A note on semipositone boundary value problems with nonlocal, nonlinear boundary conditions, \textit{Arch. Math. (Basel)}, \textbf{103} (2014),  177--187.

\bibitem{Goodrich7}
C. S. Goodrich, Semipositone boundary value problems with nonlocal, nonlinear boundary conditions, \textit{Adv. Differential Equations}, \textbf{20} (2015), 117--142.

\bibitem{Goodrich8}
C. S. Goodrich, Radially symmetric solutions of elliptic PDEs with uniformly negative weight, \textit{Ann. Mat. Pura Appl.}, \textbf{197} (2018),  1585--1611. 

\bibitem{Goodrich9}
C. S. Goodrich, New Harnack inequalities and existence theorems for radially symmetric solutions of elliptic PDEs with sign changing or vanishing Green's function,
\textit{J. Differential Equations}, \textit{264} (2018), 236--262. 

\bibitem{guolak} D. Guo and V. Lakshmikantham,
\textit{Nonlinear problems in abstract cones}, Academic Press, Boston,
1988.

\bibitem{gi-caa} G. Infante,
 Nonlocal boundary value problems with two nonlinear boundary conditions,
\textit{Commun. Appl. Anal.}, \textbf{12} (2008), 279--288.

\bibitem{gi-mc}
G. Infante, A short course on positive solutions of systems of ODEs via fixed point index, in \textit{Lecture Notes
in Nonlinear Analysis (LNNA)}, \textbf{16} (2017), 93--140.

\bibitem{gi-tmna}
G. Infante, Nonzero positive solutions of a multi-parameter elliptic system with functional BCs, \textit{Topol. Methods Nonlinear Anal.}, \textbf{52}, (2018), 665--675.

\bibitem{gijwnodea} G. Infante and J. R. L. Webb, Loss of positivity in
a nonlinear scalar heat equation, \textit{NoDEA Nonlinear
Differential Equations Appl.}, \textbf{13} (2006), 249-261.

\bibitem{gijwems} G. Infante and J. R. L. Webb,
  Nonlinear nonlocal boundary value problems and perturbed Hammerstein integral equations,
  \textit{Proc. Edinb. Math. Soc.}, \textbf{49} (2006), 637--656.

\bibitem{kamkee}
G. Kalna and S. McKee, The thermostat problem,  \textit{TEMA Tend.
Mat. Apl. Comput.}, \textbf{3 }(2002), 15--29.

\bibitem{kamkee2}
G. Kalna and S. McKee, The thermostat problem with a nonlocal
nonlinear boundary condition, \textit{IMA J. Appl. Math.},
\textbf{69} (2004), 437--462.

\bibitem{kejde}
G. L. Karakostas,
Existence of solutions for an $n$-dimensional operator equation and applications to BVPs,
\textit{Electron. J. Differential Equations}, \textbf{2014}, No. 71, 17~pp.

\bibitem{kttmna} G. L. Karakostas and P. Ch. Tsamatos, Existence of multiple
positive solutions for a nonlocal boundary value problem,
\textit{Topol. Methods Nonlinear Anal.}, \textbf{19} (2002),
109--121.

\bibitem{ktejde}
G. L. Karakostas and P. Ch. Tsamatos, Multiple positive solutions
of some Fredholm integral equations arisen from nonlocal
boundary-value problems, \textit{Electron. J. Differential
Equations}, \textbf{2002}, 17 pp.


\bibitem{kapala}
I. Karatsompanis and P. K. Palamides, Polynomial approximation to
a non-local boundary value problem, \textit{Comput. Math. Appl.},
\textbf{60} (2010),  3058--3071.

\bibitem{rma}
R. Ma, A survey on nonlocal boundary value problems, \textit{Appl.
Math. E-Notes}, \textbf{7} (2007), 257--279.

\bibitem{Mawhin}
J. Mawhin, B. Przeradzki and K. Szyma\'{n}ska-D\c{e}bowska, Second order systems with nonlinear nonlocal boundary conditions, \textit{Electron. J. Qual. Theory Differ. Equ.}, \textbf{2018}, No. 56, 1-11.

\bibitem{Nieto}
J. J. Nieto, Existence of a solution for a three-point boundary value problem for a second-order differential equation at resonance, \textit{Bound. Value Probl.}, \textbf{2013:130} (2013), 7 pp.


\bibitem{nipime}
J. J. Nieto and J. Pimentel, Positive solutions of a fractional
thermostat model, \textit{Bound. Value Probl.}, \textbf{2013:5}
(2013), 11 pp.

\bibitem{palamides}
P. Palamides, G. Infante and P. Pietramala, Nontrivial solutions
of a nonlinear heat flow problem via Sperner's Lemma,
\textit{Appl. Math. Lett.},  \textbf{22} (2009), 1444--1450.

\bibitem{paola}  P. Pietramala,
A note on a beam equation with nonlinear boundary conditions,
\textit{Bound. Value Probl.}, \textbf{2011} (2011), Art. ID
376782, 14 pp.

\bibitem{Picone} M. Picone, Su un problema al contorno nelle equazioni differenziali lineari ordinarie del secondo ordine,
\textit{Ann. Scuola Norm. Sup. Pisa Cl. Sci.}, \textbf{10} (1908), 1--95.

\bibitem{sotiris}
S. K. Ntouyas, Nonlocal initial and boundary value problems: a
survey, \textit{Handbook of differential equations: ordinary
differential equations. Vol. II}, Elsevier B. V., Amsterdam,
(2005), 461--557.

\bibitem{Stik} A. \v{S}tikonas, A survey on stationary problems, Green's functions and spectrum of Sturm-Liouville problem
with nonlocal boundary conditions, \textit{Nonlinear Anal. Model.
Control}, \textbf{19} (2014), 301--334.


\bibitem{jwpomona}
J. R. L. Webb, Multiple positive solutions of some nonlinear heat
flow problems, \textit{Discrete Contin. Dyn. Syst. (Suppl.)},
(2005), 895--903.

\bibitem{jwwcna04} J. R. L. Webb,
Optimal constants in a nonlocal boundary value problem,
 \textit{Nonlinear Anal.}, \textbf{63} (2005), 672--685.

\bibitem{webb-therm}
J. R. L. Webb, Existence of positive solutions for a thermostat
model, \textit{Nonlinear Anal. Real World Appl.}, \textbf{13}
(2012), 923--938.  

\bibitem{webb-ejqtde}
J. R. L. Webb, Positive solutions of nonlinear differential equations with Riemann-Stieltjes boundary conditions,
 \textit{Electron. J. Qual. Theory Differ. Equ.}, \textbf{2016}, No. 86, 1-13.

\bibitem{jw-gi-jlms} J. R. L. Webb and G. Infante,
Positive solutions of nonlocal boundary value problems: a unified
approach, \textit{J. London Math. Soc.}, \textbf{74} (2006),
673--693.


\bibitem{Whyburn}
W. M. Whyburn, Differential equations with general boundary
conditions, \textit{Bull. Amer. Math. Soc.}, \textbf{48} (1942),
692--704.

\bibitem{ya1} Z. Yang,
Positive solutions to a system of second-order nonlocal boundary
value problems, \textit{Nonlinear Anal.}, \textbf{62} (2005),
1251--1265.

\bibitem{ya2}
Z. Yang, Positive solutions of a second-order integral boundary
value problem, \textit{J. Math. Anal. Appl.}, \textbf{321} (2006),
751--765.

\bibitem{zang}
J. Zang, G. Zhang and H. Li, Positive solutions of second-order problem with dependence on derivative in nonlinearity under Stieltjes integral boundary condition, 
 \textit{Electron. J. Qual. Theory Differ. Equ.}, \textbf{2018}, No. 4, 1-13.
\end{thebibliography}
\end{document}